\documentclass{amsart}
\usepackage{bm}
\usepackage{amssymb}
\usepackage{enumerate}
\usepackage{xcolor}
\usepackage{commath}
\usepackage{amsfonts}
\usepackage{amsmath,amsthm}
\usepackage{dsfont}
\usepackage{hyperref}
\usepackage{algorithm}
\usepackage{framed}
\usepackage{boites}
\usepackage{mathrsfs}
\usepackage[noend]{algpseudocode}
\usepackage{graphicx}
\usepackage{caption}
\usepackage{subcaption}
\usepackage{booktabs}
\usepackage{amsmath}
\usepackage{natbib}
\usepackage{algorithm}
\usepackage{algpseudocode}
\usepackage{url}

\newcommand{\arginf}{\mathop{\rm arg~inf}\limits}

\usepackage{subcaption}

\newtheorem{thm}{Theorem}

\newtheorem{prop}[thm]{Proposition}

\newtheorem{corollary}[thm]{Corollary}

\newtheorem{assumption}{Assumption}
\theoremstyle{definition}

\theoremstyle{remark}

\usepackage{fancyhdr}

\subjclass[2010]{35R30, 58J50, 65F10, 65N12}

\title{OCEAN WAVE SPECTRUM RECONSTRUCTION FROM HF RADAR DATA AND ITS APPLICATION TO WAVE HEIGHT ESTIMATION}
\author{Kaede Watanabe}
\address{Department of Mathematics, Graduate School of Science, Tohoku University,
6-3 Aoba, Aramaki-aza, Aoba-ku, Sendai, Miyagi 980-8578, Japan. \& Advanced Institute for Materials Research (AIMR), Tohoku University, 2-1-1 Katahira, Aoba-ku, Sendai, Miyagi 980-8577, Japan.}
\email{kaede.watanabe.t3@dc.tohoku.ac.jp}

\author{Toshiaki Yachimura}
\address{Mathematical Science Center for Co-creative Society,
Tohoku University, 468-1 Aoba, Aramaki, Aoba-ku, Sendai, Miyagi 980-0845, Japan .}
\email{toshiaki.yachimura.a4@tohoku.ac.jp}

\author{Tsubasa Terada}
\address{Information Technology R\&D Center, Mitsubishi Electric Corporation, 5-1-1 Ofuna, Kamakura, Kanagawa 247-8501, Japan.}
\email{Terada.Tsubasa@cb.mitsubishielectric.co.jp}

\author{Hiroshi Kameda}
\address{Information Technology R\&D Center, Mitsubishi Electric Corporation, 5-1-1 Ofuna, Kamakura, Kanagawa 247-8501, Japan.}
\email{Kameda.Hiroshi@ab.mitsubishielectric.co.jp}

\author{Ryuhei Takahashi}
\address{Information Technology R\&D Center, Mitsubishi Electric Corporation, 5-1-1 Ofuna, Kamakura, Kanagawa 247-8501, Japan .}
\email{Takahashi.Ryuhei@ab.mitsubishielectric.co.jp}

\author{Hiroshi Suito}
\address{Advanced Institute for Materials Research (AIMR), Tohoku University, 2-1-1 Katahira, Aoba-ku, Sendai, Miyagi 980-8577, Japan. \& Mathematical Science Center for Co-creative Society,
Tohoku University, 468-1 Aoba, Aramaki, Aoba-ku, Sendai, Miyagi 980-0845, Japan .}
\email{hiroshi.suito@tohoku.ac.jp}

\subjclass[2010]{45Q05, 65J20, 65N12, 65R32, 86A05}

\keywords{Ocean wave height estimation, remote sensing, optimization problem, Tikhonov regularization, nonnegative sparse regularization, iterative method}

\begin{document}

\begin{abstract}
Real-time estimation of ocean wave heights using high-frequency (HF) radar has attracted great attention. This method offers the benefit of easy maintenance by virtue of its ground-based installation. However, it is adversely affected by issues such as low estimation accuracy. As described herein, we propose an algorithm based on the nonnegative sparse regularization method to estimate the energy distribution of the component waves, known as the ocean wave spectrum, from HF radar data. After proving a stability estimate of this algorithm, we perform numerical simulations to verify the proposed method's effectiveness.
\end{abstract}

\maketitle

\section{Introduction}\label{Introduction}
Accurate and rapid estimation of ocean wave heights is necessary to mitigate damage caused by high waves associated with typhoons or earthquakes. Additionally, such estimations are crucially important for industrial applications, including port construction, disaster prevention, and fisheries. Wave height estimation can be performed using devices such as buoys and using GPS-based methods \cite{Breunung2024, Joodaki2013}. However, in recent years, measurement techniques using high-frequency (HF) radar have attracted great attention. One reason is that HF radar is installed along the coast, with attendant ease of deployment and maintenance on land. Furthermore, wave conditions in coastal waters are often complex. The accuracy of GPS-based wave height measurements tends to be lower there \cite{Cavaleri2018}. Therefore, for such areas, HF radar is anticipated as a promising tool for wave height estimation.

HF radar is a system that transmits electromagnetic waves toward the sea and receives the reflected waves. The obtained electromagnetic waves are analyzed using spectral analysis, producing a so-called Doppler spectrum (Figure~\ref{fig:example}).

\begin{figure}[t]
    \centering
    \includegraphics[width=10.0cm]{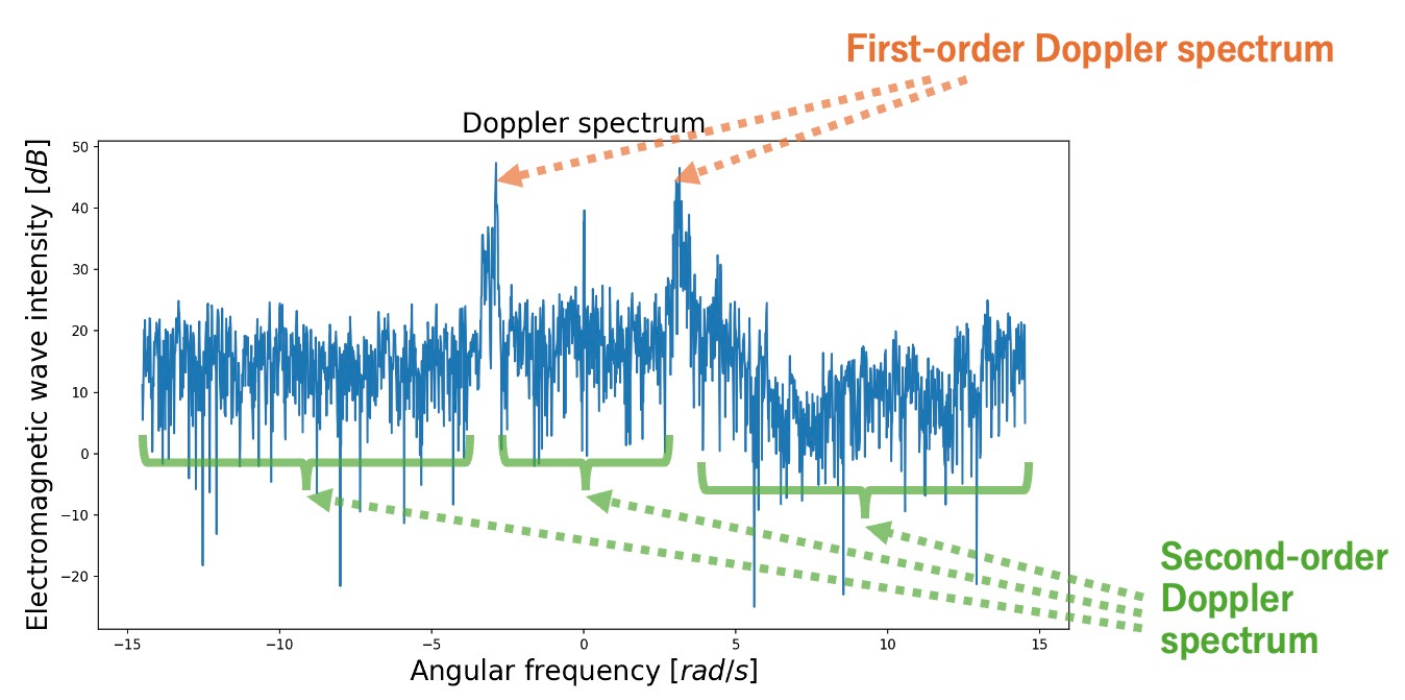} 
    \caption{Example of a Doppler spectrum observed using HF radar. The peaks appearing symmetrically on both sides correspond to the first-order Doppler spectrum, whereas the remaining parts correspond to the second-order Doppler spectrum.}
    \label{fig:example}
\end{figure}
Ocean waves are represented as a superposition of component waves with different directions and frequencies \cite{Massie2001}. The Doppler spectrum reflects the information of component waves. Two prominent peaks can be observed from it. This phenomenon occurs when electromagnetic waves with a wavenumber vector $ \bm{k}_0 = (k_0, 0) $ are scattered by component waves with wavenumber vectors $ \pm 2\bm{k}_0 $. This phenomenon is known as first-order scattering. The remaining parts are designated as the second-order Doppler spectrum. These occur when the electromagnetic wave is scattered by the component waves $ \bm{k}_1 $ and $ \bm{k}_2 $ satisfying $ \bm{k}_1 + \bm{k}_2 = -2\bm{k}_0 $ \cite{lipa1982}. Here, since $ \bm{k}_1 $ can take any wavenumber vector, we can infer that the second-order Doppler spectrum includes information about all component waves. The following relation is known to hold between the second-order Doppler spectrum and the ocean wave spectrum, which expresses the energy distribution of the component waves \cite{Johnstone1975} as
\begin{align}
\label{main1}
\sigma_2(\omega) = 2^6 \pi k_0^4 &\sum_{m_1 = \pm 1} \sum_{m_2 = \pm 1} \int_{-\infty}^{\infty} \int_{-\infty}^{\infty} |\Gamma(m_1, m_2, \omega, \bm{k}_1, \bm{k}_2)|^2 \nonumber \\
& S(m_1\bm{k}_1)S(m_2\bm{k}_2) \delta(\omega - m_1\sqrt{g k_1} - m_2\sqrt{g k_2}) \, dp \, dq, 
\end{align}
where $\sigma_2(\omega)$ represents the second-order Doppler spectrum, and where $ S(p, q) $ denotes the ocean wave spectrum. The wavenumber vectors are described as $ \bm{k}_1 = (p - k_0, q) $ and $ \bm{k}_2 = (-p - k_0, -q) $, satisfying $ \bm{k}_1 + \bm{k}_2 = -2\bm{k}_0 $. Here, the lengths of $ \bm{k}_1 $ and $ \bm{k}_2 $ are represented respectively as $ k_1 = |\bm{k}_1| $ and $ k_2 = |\bm{k}_2| $, and $ g $ stands for gravitational acceleration. The term $ \Gamma(m_1, m_2, \omega, \bm{k}_1, \bm{k}_2) $, designated as the coupling coefficient, is defined as $ \Gamma = \Gamma_E - i \Gamma_H $, where
\[
\Gamma_E = \frac{1}{2}\left\{\frac{\frac{(\bm{k}_1 \cdot \bm{k}_0)(\bm{k}_2 \cdot \bm{k}_0)}{k_0^2} - 2\bm{k}_1 \cdot \bm{k}_2}{(\bm{k}_1 \cdot \bm{k}_2)^{1/2} - k_0 \Delta}\right\},
\]
\[
\Gamma_H = \frac{1}{2} \left\{ k_1 + k_2 + \frac{k_1 k_2 - \bm{k}_1 \cdot \bm{k}_2}{m_1 m_2 \left( k_1 k_2 \right)^{\frac{1}{2}}}\cdot \frac{2gk_0 + \omega^2}{2gk_0 - \omega^2} \right\}.
\]
 Here, $ \Delta $ represents the ocean impedance, defined as $ \Delta = 0.011 - 0.012i $  \cite{lipa1982}. Also, $ \delta $ denotes the delta function. 

The significant wave concept, which is devised to represent the complex conditions of the sea surface, serves as an indicator of the average wave state at a given location. This wave height is designated as the significant wave height $H_s$. Its period is designated as the significant wave period. The relation between $H_s$ and the wave spectrum $S(p, q)$ is given as
\begin{align}
\label{height}
    H_s = 4\sqrt{\int_{\mathbb{R}^2} S(p, q) \, dp \, dq}
\end{align}
based on the formulation in \cite{Justin2008}.

Several approaches have been proposed for estimating significant wave heights from HF radar data: the Barrick method is one such approach \cite{Barrick1977a, Barrick1977b}. This method uses the approximate equation obtained from \eqref{main1} between the Doppler spectra and the frequency spectrum, which is derived by integrating the wave spectrum over angles. By applying this relation with \eqref{height}, significant wave heights can be computed. An extended method has also been proposed by which the relation between the Doppler spectrum and the significant wave height is modified \cite{Maresca1980, Heron1998, Ramos2009}.

For estimation of the wave spectrum, various approaches use both the second-order Doppler spectrum and the relation expressed in \eqref{main1}. Some techniques reformulate \eqref{main1} as a linear equation by imposing constraints on $S$ \cite{Lipa1977, Lipa1978}. More recently, methods for solving the nonlinear equation directly without resorting to linearization have been proposed \cite{Hisaki1996, hashimoto1999, hashimoto2003, Wyatt2006, Hisaki2015, kataoka2016, Guiomar2019}.

Various methods exist for estimating the wave spectrum from the second-order Doppler spectrum. However, very few studies have investigated the problem in a mathematically rigorous manner.

As described herein, we aim at formulating the inverse problem of estimating the wave spectrum $S$ from the second-order Doppler spectrum $\sigma_2$ using regularization methods. After investigating it mathematically, we construct the algorithm to solve this inverse problem and verify it. Additionally, we evaluate its stability and apply the algorithm to actual observed data.

The paper is organized as described hereinafter. In Section \ref{Preliminaries}, we formulate the inverse problem using nonnegative sparse regularization with the Tikhonov functional. To this end, we eliminate the delta function from \eqref{main1} and prove that the second-order Doppler spectrum belongs to $L^2$ when the ocean wave spectrum is an element of $H^1(\mathbb{R}^2)$. Furthermore, we provide a brief explanation of the nonnegative sparse regularization approach. In Section \ref{derivative and stability}, we derive the explicit form of the Fr\'echet derivative of the Tikhonov functional, which is used to construct a numerical algorithm for solving the inverse problem defined by \eqref{main1}. Additionally, we present a stability theorem for the inverse problem. As described in Section \ref{Numerical experiment}, we perform numerical experiments. First, we introduce the model wave spectrum. Then, after assessing the performance of the reconstruction, we examine its stability through numerical experimentation. Finally, we apply the algorithm to actual observed data to evaluate its effectiveness for wave height estimation.

\section{Preliminaries}\label{Preliminaries}
We denote the right-hand side of \eqref{main1} as $A$. In other words,
\begin{align}
A[S] := 2^6 \pi k_0^4 &\sum_{m_1=\pm1} \sum_{m_2=\pm1} \int_{\mathbb{R}^2} \left|\Gamma(m_1, m_2, \omega, \bm{k_1}, \bm{k_2})\right|^2 \nonumber\\
&S(m_1 \bm{k_1}) S(m_2 \bm{k_2}) \delta(\omega - m_1 \sqrt{g k_1} - m_2 \sqrt{g k_2}) dpdq. \nonumber
\end{align}
Let $ K $ be the domain of $ \omega $. Also, assume that $ K $ satisfies the following conditions.
\begin{enumerate}
    \item The set $K$ is compact.
    \label{as1}
    \item $\pm\sqrt{2gk_0}, \, \pm2\sqrt{gk_0} \notin K$.
    \label{as2}
\end{enumerate}
Under the conditions presented above, $\Gamma$ can be regarded as $\Gamma \in C(K \times \mathbb{R}^2)$ for fixed $m_1, m_2$.

\subsection{Elimination of the delta function}
For the convenience of later analysis, the delta function appearing in the integral of \eqref{main1} is eliminated. With the setting above, the following proposition holds.
\begin{prop}
\label{prop1}
Assume $S \in H^1(\mathbb{R}^2)$. Then, the operator $A$ can be expressed as  
\begin{align}
\label{main}
    A[S](\omega) = 2^6 \pi k_0^4 \sum_{m_1=\pm1} \sum_{m_2=\pm1} \int_{f_{m_1, m_2}^{-1}(\omega)} \left| \Gamma(m_1, m_2, \omega, \bm{k_1}, \bm{k_2}) \right|^2 
    \frac{S(m_1 \bm{k_1}) S(m_2 \bm{k_2})}{\left| \nabla f_{m_1, m_2}(p, q) \right|} \, ds.
\end{align}  
Therein, $f_{m_1, m_2}(p, q) :=m_1 \sqrt{gk_1} + m_2 \sqrt{gk_2}$, and $ds$ is the surface measure on $f_{m_1,m_2}^{-1}(\omega)$. Furthermore, $A: H^1(\mathbb{R}^2) \to L^2(K)$.
\end{prop}
\begin{proof}
Consider $\mathbb{R}^2 \setminus \{(\pm k_0, 0), (0, 0)\}$. In this domain, the function 
\begin{align}
    f_{m_1, m_2}(p, q) &= m_1 \sqrt{gk_1} + m_2 \sqrt{gk_2}\notag\\
    &=m_1\sqrt{g}\{(p-k_0)^2+q^2\}^{\frac{1}{4}}+m_2\sqrt{g}\{(p+k_0)^2+q^2\}^{\frac{1}{4}}\notag
\end{align}
is of class $C^\infty$ and $f_{m_1, m_2} \neq 0$. From
\begin{align}
    \frac{\partial f_{m_1, m_2}}{\partial p} 
    &= -\frac{1}{2} m_1  \sqrt{g} (p - k_0) \{ (p - k_0)^2 + q^2 \}^{-\frac{3}{4}} 
    - \frac{1}{2} m_2  \sqrt{g} (p + k_0) \{ (p + k_0)^2 + q^2 \}^{-\frac{3}{4}}, \notag \\
    \frac{\partial f_{m_1, m_2}}{\partial q} 
    &= -\frac{1}{2} m_1  \sqrt{g} q \{ (p + k_0)^2 + q^2 \}^{-\frac{3}{4}} 
    - \frac{1}{2} m_2  \sqrt{g} q \{ (p + k_0)^2 + q^2 \}^{-\frac{3}{4}},\notag
\end{align}
$f_{m_1, m_2}$ satisfies $|\nabla f_{ m_1, m_2}| \neq 0$. Thus, applying [\citealp{hörmander2015analysis}, Theorem 6.1.5, p.136] to $\delta(\omega - f_{m_1,m_2}(p,q))=|\nabla f_{m_1,m_2}(p,q)|^{-1}ds$, we obtain the result presented below:
\begin{align}
    A[S](\omega) = 2^6 \pi k_0^4 \sum_{m_1=\pm1} &\sum_{m_2=\pm1} \int_{\mathbb{R}^2\setminus \{(\pm k_0, 0), (0, 0)\}} \left|\Gamma(m_1, m_2, \omega, \bm{k_1}, \bm{k_2})\right|^2 \nonumber\\
    &S(m_1 \bm{k_1}) S(m_2 \bm{k_2}) \delta(\omega - m_1 \sqrt{g k_1} - m_2 \sqrt{g k_2}) dpdq \nonumber\\
    =2^6 \pi k_0^4 \sum_{m_1=\pm1} &\sum_{m_2=\pm1} \int_{f_{ m_1, m_2}^{-1}(\omega)} \left| \Gamma(m_1, m_2, \omega, \bm{k_1}, \bm{k_2}) \right|^2 
    \frac{S(m_1 \bm{k_1}) S(m_2 \bm{k_2})}{\left| \nabla f_{m_1, m_2}(p, q) \right|} \, ds.\nonumber
\end{align}
For any $ \omega \in K \cap(-\sqrt{2gk_0}, \sqrt{2gk_0})$, 
\begin{align}
    |A[S](\omega)| &=C\left|\sum_{m_1=\pm1} \sum_{m_2=\pm1}\int_{f_{m_1, m_2}^{-1}(\omega)} \left| \Gamma(m_1, m_2, \omega, \bm{k_1}, \bm{k_2}) \right|^2 
        \frac{S(m_1 \bm{k_1}) S(m_2 \bm{k_2})}{\left| \nabla f_{m_1, m_2}(p, q) \right|} \, ds\right|\notag\\
        &\leq C \sum_{m_1=\pm1} \sum_{m_2=\pm1}\int_{f_{ m_1, m_2}^{-1}(\omega)} \frac{|\Gamma|^2 }{|\nabla f_{ m_1, m_2}|} |S(m_1 \bm{k_1})| |S(m_2 \bm{k_2})|\, ds\notag
\end{align}
holds. Here, we put $C= 2^6 \pi k_0^4$. Since $ f_{m_1,m_2} $ is a continuous function, the inverse image $ f_{m_1,m_2}^{-1}(K) = \{(p,q) \in \mathbb{R}^2 \, | \, m_1\sqrt{g}\{(p-k_0)^2+q^2\}^{\frac{1}{4}}+m_2\sqrt{g}\{(p+k_0)^2+q^2\}^{\frac{1}{4}}\in K\} $ is a closed and bounded set. 
Thus, $ |\Gamma|^2 / |\nabla f_{m_1,m_2}| $ has its maximum value $M$ on $K\times f_{m_1,m_2}^{-1}(K) $. Then, we have
\begin{align}
    |A[S](\omega)| 
    &\leq CM \sum_{m_1=\pm1} \sum_{m_2=\pm1}\int_{f_{m_1, m_2}^{-1}(\omega)} |S(m_1 \bm{k_1})| |S(m_2 \bm{k_2})| \, ds.\notag
\end{align}
Since $ f_{m_1, m_2}^{-1}(\omega) $ is bounded for any $ \omega \in K$, we define the domain enclosed by the curve $ f_{m_1, m_2}^{-1}(\omega) $ as $ \Omega $. By applying H\"older's inequality and the trace theorem (see \citep[Theorem 1, p.258]{evans2010} and \citep[pp.315--316]{brezis2010functional}), we obtain
\begin{align}
    |A[S](\omega)| &\leq C M \sum_{m_1=\pm1} \sum_{m_2=\pm1}\int_{f_{m_1, m_2}^{-1}(\omega)} |S(m_1 \bm{k_1})| |S(m_2 \bm{k_2})| \, ds\notag\\
    &\leq C M\sum_{m_1=\pm1} \sum_{m_2=\pm1}\|S(m_1 \bm{k_1})\|_{L^2(\partial\Omega)} \|S(m_2 \bm{k_2})\|_{L^2(\partial\Omega)}\nonumber\\
    &\leq C M\sum_{m_1=\pm1} \sum_{m_2=\pm1}C_1\|S(m_1 \bm{k_1})\|_{H^1(\Omega)} C_2 \|S(m_2 \bm{k_2})\|_{H^1(\Omega)}\nonumber\\
    &\leq C M\sum_{m_1=\pm1} \sum_{m_2=\pm1}C_1\|S(m_1 \bm{k_1})\|_{H^1(\mathbb{R}^2)} C_2 \|S(m_2 \bm{k_2})\|_{H^1(\mathbb{R}^2)}\nonumber\\
    &\leq L_1\|S\|^2_{H^1(\mathbb{R}^2)},\notag
\end{align}
where $C_1,C_2,L_1>0$ are constants. The same argument is applied for the case in which $ \omega \in K \cap \left(-\sqrt{2gk_0}, \sqrt{2gk_0}\right)^c $. Therefore, we infer $ A[S] \in L^2(K) $ for any $ S \in H^1(\mathbb{R}^2) $.

\end{proof}

\subsection{Nonnegative sparse regularization}
To address the inverse problem from the second-order Doppler spectrum to the wave spectrum, the fact that the wave spectrum takes nonnegative values must be considered. Therefore, we adopt the nonnegative sparse regularization method proposed by \cite{Muoi2018}.

We introduce the general form of the nonnegative sparse regularization. Using this method, the following minimization problem for the observed data $y^\delta \in Y$,
\begin{align}
\label{main2}
   \arginf_{S \in X}\Theta(S) := F(S; y^\delta) + \alpha \Phi(S)
\end{align}
is considered, where $X,Y$ are Hilbert spaces and $\alpha > 0$. In this method, the solution of \eqref{main2} is defined as the approximate solution of the inverse problem. We assume the following conditions for $ F(\cdot; y^\delta) $.
\begin{assumption}
\label{assu1}
    Let $ X $ and $ Y $ be real Hilbert spaces and let $ \{ \varphi_k \}_{k \in \Lambda} $ be an orthonormal basis of $ X $.
    \begin{enumerate}
        \item $ \operatorname{Dom}(F(\cdot; y^\delta)) = X $ holds for any $ y^\delta \in Y $, and for each $ y^\delta \in Y $, $ F(\cdot; y^\delta) $ is bounded from below and weakly lower semicontinuous. Without loss of generality, we assume $ F(u; y^\delta) \geq 0 $ for all $ u \in X $.

        \item $F(\cdot; y^\delta)$ is Fr\'echet differentiable on $X$.
    \end{enumerate}
\end{assumption}
\noindent
Let $\{u_i\}$ be the coefficients of $u \in X$ when represented using the basis $\{\varphi_i\}$. Then, $\Phi(u)$ is given as
\[
\Phi(u) = \begin{cases}
\sum_{i \in \Lambda} u_{i}, & \text{if } u \geq 0, \\
+\infty, & \text{otherwise}.
\end{cases}
\]

As described herein, the observed data $y^\delta$ correspond to the second-order Doppler spectrum $\sigma_2 \in L^2(K)$. Assuming $X = H^1(\mathbb{R}^2)$, and letting $S\in H^1(\mathbb{R}^2)$ denote the wave spectrum, then the Tikhonov functional is defined as 
\[
F(S;\sigma_2) = \|A[S] - \sigma_2\|_{L^2(K)}^2 + \lambda \|S\|_{L^2(\mathbb{R}^2)}^2,
\]
where $\lambda > 0$ is the regularization parameter. For this study, we define $\Phi(S)$ as 
\[
\Phi(S) = \begin{cases}
\sum_{i,j \in \Lambda} S_{i,j}, & \text{if } S \geq 0, \\
+\infty, & \text{otherwise}.
\end{cases}
\]
Therein, $S_{i,j}$ represents the $(i,j)$-th component of discretized $S$ with an $N \times N$ Cartesian grid, where $N$ is a natural number and $\Lambda$ is the subscript set. This term ensures that $S \geq 0$, preventing $\Theta$ from being minimized otherwise; this term also promotes sparsity. Regarding the functional $ \Theta $ in \eqref{main2}, the result is known.

\begin{thm}[\mbox{\citep[Theorem 2.1, p.4]{Muoi2018}}]
\label{the2}
    Under Assumption \ref{assu1}, the following holds:
    \begin{enumerate}
        \item The minimization problem \eqref{main2} has at least one nonnegative solution. In other words, the functional $\Theta$ has at least one nonnegative global minimizer.

        \item Any local minimizer $u$ of $\Theta$ satisfies
        \[
        u = \mathbb{P}_{s \alpha} \left( u - s \nabla F(u; y^\delta) \right),
        \]
        for any $s > 0$, where the proximity operator $ \mathbb{P}_{\xi} $ is defined as
        \[
        \mathbb{P}_{\xi}(u) = \sum_{i \in \Lambda} \max(u_i - \xi, 0) \varphi_i,
        \]
        with $u_i = \langle u, \varphi_i \rangle_X$.

        \item If $u$ is a local minimizer of the functional $\Theta$, then $u$ is nonnegative and sparse. Here, sparsity means that when $u$ is expressed using the orthonormal basis $\{ \varphi_k \}_{k \in \Lambda}$, only a finite number of coefficients are nonzero, whereas all others are zero.

    \end{enumerate}
\end{thm}

\section{Fr\'echet derivative of Tikhonov functional and stability estimate}\label{derivative and stability}
To develop a numerical algorithm, we compute the Fr\'echet derivative of $F(S; \sigma_2)$. For this purpose, we first calculate the Fr\'echet derivative of $A[S]$.
\begin{prop}
\label{prop3}
    Map $A[S]$ is Fr\'echet differentiable. Its Fr\'echet derivative at $ S $ is given as
    \begin{align}
    \label{A_derivative}
        A'[S]h = 2^7 \pi k_0^4 \sum_{m_1 = \pm 1} \sum_{m_2 = \pm 1} \int_{f_{ m_1, m_2}^{-1}(\omega)} \frac{|\Gamma|^2}{|\nabla f_{m_1, m_2}|} S(m_1 \bm{k_1}) h(m_2 \bm{k_2}) ds,
    \end{align}
where $h \in H^1(\mathbb{R}^2)$.
\end{prop}

\begin{proof}
    Let the operator $T$ be given by the right-hand side of \eqref{A_derivative}. For any $\omega \in K$,
    \begin{align}
    |Th(\omega)|= \bigg|2^7 \pi k_0^4 \sum_{m_1 = \pm 1} \sum_{m_2 = \pm 1} \int_{f_{ m_1, m_2}^{-1}(\omega)} \frac{|\Gamma|^2}{|\nabla f_{m_1, m_2}|} S(m_1 \bm{k_1}) h(m_2 \bm{k_2}) ds\bigg|\notag\\
    \le 2^7 \pi k_0^4 \sum_{m_1 = \pm 1} \sum_{m_2 = \pm 1} \int_{f_{ m_1, m_2}^{-1}(\omega)} \frac{|\Gamma|^2}{|\nabla f_{m_1, m_2}|} |S(m_1 \bm{k_1})| |h(m_2 \bm{k_2})| ds\notag.
\end{align}
As in Proposition \ref{prop1}, the set $K\times f_{m_1,m_2}^{-1}(K)$ is closed and bounded. Thus, $|\Gamma|^2/|\nabla f_{m_1,m_2}|$ has its maximum value $M$. Applying a similar transformation as in Proposition \ref{prop1}, we obtain $|Th(\omega)|\le L_2\|h\|_{H^1(\mathbb{R}^2)}$, where $L_2 > 0$ is a constant independent of $\omega$. Thus, $T$ is a bounded linear operator.

Next, we prove the differentiability of $A$. The difference $A[S+h]-A[S]$ is 
\begin{align}
    A[S + h] - A[S] 
    &= 2^6 \pi k_0^4\sum_{m_1 = \pm 1} \sum_{m_2 = \pm 1}  \int_{f_{ m_1, m_2}^{-1}(\omega)} \frac{|\Gamma|^2}{|\nabla f_{m_1, m_2}|}
    \big\{S(m_1 \bm{k_1}) + h(m_1 \bm{k_1})\big\} \notag \\
    &\phantom{=} \quad \times \big\{S(m_2 \bm{k_2}) + h(m_2 \bm{k_2})\big\} \, ds \notag \\
    &\phantom{=} \quad - 2^6 \pi k_0^4 \sum_{m_1 = \pm 1} \sum_{m_2 = \pm 1} \int_{f_{m_1, m_2}^{-1}(\omega)} \frac{|\Gamma|^2}{|\nabla f_{m_1, m_2}|} 
    S(m_1 \bm{k_1}) S(m_2 \bm{k_2}) \, ds \notag \\
    &= 2^6 \pi k_0^4 \sum_{m_1 = \pm 1} \sum_{m_2 = \pm 1} \int_{f_{ m_1, m_2}^{-1}(\omega)} \frac{|\Gamma|^2}{|\nabla f_{m_1, m_2}|} 
    \big\{S(m_1 \bm{k_1}) h(m_2 \bm{k_2}) + \notag \\
    &\phantom{=} \quad S(m_2 \bm{k_2}) h(m_1 \bm{k_1})\big\} \, ds \notag \\
    &\phantom{=} \quad + 2^6 \pi k_0^4 \sum_{m_1 = \pm 1} \sum_{m_2 = \pm 1} \int_{f_{m_1, m_2}^{-1}(\omega)} \frac{|\Gamma|^2}{|\nabla f_{m_1, m_2}|} 
    h(m_1 \bm{k_1}) h(m_2 \bm{k_2}) \, ds \notag \\
    &=2^6 \pi k_0^4 \sum_{m_1 = \pm 1} \sum_{m_2 = \pm 1} \int_{f_{ m_1, m_2}^{-1}(\omega)} \frac{|\Gamma|^2}{|\nabla f_{m_1, m_2}|} 
    S(m_1 \bm{k_1}) h(m_2 \bm{k_2}) ds\notag \\
    &\phantom{=} \quad + 2^6 \pi k_0^4 \sum_{m_1 = \pm 1} \sum_{m_2 = \pm 1} \int_{f_{ m_1, m_2}^{-1}(\omega)} \frac{|\Gamma|^2}{|\nabla f_{m_1, m_2}|} S(m_2 \bm{k_2}) h(m_1 \bm{k_1}) ds\notag \\
    &\phantom{=} \quad + 2^6 \pi k_0^4 \sum_{m_1 = \pm 1} \sum_{m_2 = \pm 1} \int_{f_{m_1, m_2}^{-1}(\omega)} \frac{|\Gamma|^2}{|\nabla f_{m_1, m_2}|} 
    h(m_1 \bm{k_1}) h(m_2 \bm{k_2}) \, ds\notag\\
    &=: I + II + III. \notag
\end{align}
By a change of variables $ p \mapsto -p $ and $ q \mapsto -q $,
\begin{align}
    I=2^6 \pi k_0^4 \sum_{m_1 = \pm 1} \sum_{m_2 = \pm 1} &\int_{f_{ m_1, m_2}^{-1}(\omega)} \frac{|\Gamma(m_1, m_2, \omega, \bm{k_1}, \bm{k_2})|^2}{|\nabla f_{m_1, m_2}|}
    S(m_1 \bm{k_1}) h(m_2 \bm{k_2}) ds\notag \\
    =2^6 \pi k_0^4\sum_{m_1=\pm1} \sum_{m_2=\pm1} &\int_{f_{ m_1, m_2}^{-1}(\omega)} \frac{\left|\Gamma(m_1, m_2, \omega, \bm{k_2}, \bm{k_1})\right|^2}{|\nabla f_{m_1, m_2}|} S(m_1 \bm{k_2}) h(m_2 \bm{k_1}) ds.\notag
\end{align}
From the symmetry of $\Gamma$ with respect to $\bm{k_1}$ and $\bm{k_2}$,
\begin{align}
    I=&2^6 \pi k_0^4\sum_{m_1=\pm1} \sum_{m_2=\pm1} \int_{f_{ m_1, m_2}^{-1}(\omega)} \frac{\left|\Gamma(m_1, m_2, \omega, \bm{k_2}, \bm{k_1})\right|^2}{|\nabla f_{m_1, m_2}|} S(m_1 \bm{k_2}) h(m_2 \bm{k_1}) ds\notag\\
    =&2^6 \pi k_0^4\sum_{m_1=\pm1} \sum_{m_2=\pm1} \int_{f_{ m_1, m_2}^{-1}(\omega)} \frac{\left|\Gamma(m_1, m_2, \omega, \bm{k_1}, \bm{k_2})\right|^2}{|\nabla f_{m_1, m_2}|} S(m_1 \bm{k_2}) h(m_2 \bm{k_1}) ds\notag\\
    =&II\notag.
\end{align}
Thus, $A[S+h]-A[S]=2I+III=Th+III$. Using a similar argument to that of Proposition \ref{prop1}, we can prove $III \le L_3\|h\|^2_{H^1(\mathbb{R}^2)}$, where $L_3>0$ is a constant. It follows from the above that $A[S+h]-A[S]-Th=o(\|h\|_{H^1(\mathbb{R}^2)})$, which implies that $T$ is the Fr\'echet derivative of $A$ at $S$.
\end{proof}

By applying Proposition \ref{prop3}, the explicit form of the Fr\'echet derivative of $ F $ and its gradient can be obtained.
\begin{corollary}
\label{cor4}
    The functional $F$ is Fr\'echet differentiable.
\end{corollary}

\begin{proof}
Since $A$ is Fr\'echet differentiable from Proposition \ref{prop3},
\begin{align}
    &F(S + h; \sigma_2) - F(S; \sigma_2)\notag\\
    &= \int_K \big\{(A[S + h] - \sigma_2)^2 - (A[S] - \sigma_2)^2\big\} \, d\omega + \lambda \int_{\mathbb{R}^2} \big\{(S + h)^2 - S^2\big\} \, dp dq \notag \\
    &= \int_K \big\{A[S + h] - A[S]\big\}\big\{A[S + h] + A[S] - 2\sigma_2\big\} \, d\omega \notag \\
    &+ 2 \lambda \int_{\mathbb{R}^2} S h \, dp dq + \lambda \int_{\mathbb{R}^2} h^2 \, dp dq\notag\\
    &= \int_K \big\{A'[S] h + o(\|h\|_{H^1(\mathbb{R}^2)})\big\}\big\{A'[S] h + 2A[S] + o(\|h\|_{H^1(\mathbb{R}^2)}) - 2\sigma_2\big\} \, d\omega \notag \\
    &+ 2 \lambda \int_{\mathbb{R}^2} S h \, dp dq + o(\|h\|_{H^1(\mathbb{R}^2)}).\notag
\end{align}
By a similar argument to that of Proposition \ref{prop1}, for any $ \omega \in K \cap(-\sqrt{2gk_0},\sqrt{2gk_0})$,
\begin{align}
    |A'[S]h(\omega)| &\leq \Tilde{C} M \sum_{m_1=\pm1} \sum_{m_2=\pm1}\int_{f_{m_1, m_2}^{-1}(\omega)} |S(m_1 \bm{k_1})| |h(m_2 \bm{k_2})| \, ds\notag\\
    &\leq \Tilde{C} M\sum_{m_1=\pm1} \sum_{m_2=\pm1}\|S(m_1 \bm{k_1})\|_{L^2(\partial\Omega)} \|h(m_2 \bm{k_2})\|_{L^2(\partial\Omega)}\nonumber\\
    &\leq \Tilde{C} M\sum_{m_1=\pm1} \sum_{m_2=\pm1}\Tilde{C_1}\|S(m_1 \bm{k_1})\|_{H^1(\Omega)} \Tilde{C_2} \|h(m_2 \bm{k_2})\|_{H^1(\Omega)}\nonumber\\
    &\leq \Tilde{C} M\sum_{m_1=\pm1} \sum_{m_2=\pm1}\Tilde{C_1}\|S(m_1 \bm{k_1})\|_{H^1(\mathbb{R}^2)} \Tilde{C_2} \|h(m_2 \bm{k_2})\|_{H^1(\mathbb{R}^2)}\nonumber\\
    &\leq L_4\|S\|_{H^1(\mathbb{R}^2)}\|h\|_{H^1(\mathbb{R}^2)},\notag
\end{align}
where $\Tilde{C}=2^7\pi k_0^4,\Tilde{C_1},\Tilde{C_2},M,L_4>0$ are constants. We can apply the same argument for the case in which $ \omega \in K \cap \left(-\sqrt{2gk_0}, \sqrt{2gk_0}\right)^c $. Therefore, noting that $ (A'[S]h)^2 = o(\|h\|_{H^1(\mathbb{R}^2)}) $, 
\begin{align}
\label{F_derivative}
    F(S + h; \sigma_2) - F(S; \sigma_2) &= 2 \int_K A'[S] h (A[S] - \sigma_2) \, d\omega + 2 \lambda \int_{\mathbb{R}^2} S h \, dp dq+o(\|h\|_{H^1(\mathbb{R}^2)})\notag\\
    &=:IV + V +o(\|h\|_{H^1(\mathbb{R}^2)}).
\end{align}
The operator $\Tilde{T}$ is set as $IV + V$. By applying the inequality $
|A'[S]h| \le L_4\|S\|_{H^1(\mathbb{R}^2)}\|h\|_{H^1(\mathbb{R}^2)}$ and by noting that $S \in H^1(\mathbb{R}^2)$ and $A[S],\sigma_2 \in L^2(K)$, we obtain the following.
\begin{align}
    |\Tilde{T}h|&=\bigg|2 \int_K A'[S] h (A[S] - \sigma_2) \, d\omega + 2 \lambda \int_{\mathbb{R}^2} S h \, dp dq\bigg|\notag\\
    &\le2\int_K |A'[S] h (A[S] - \sigma_2)| \, d\omega + 2 \lambda \int_{\mathbb{R}^2} |S h| \, dp dq\notag\\
    &\le2L_4
    \|S\|_{H^1(\mathbb{R}^2)}\|h\|_{H^1(\mathbb{R}^2)}\int_K |A[S] - \sigma_2|\, d\omega+ 2 \lambda\|S\|_{L^2(\mathbb{R}^2)}\|h\|_{L^2(\mathbb{R}^2)}\notag\\&\le2L_4\|S\|_{H^1(\mathbb{R}^2)}\|h\|_{H^1(\mathbb{R}^2)} \cdot L_5\|A[S] - \sigma_2\|_{L^2(\mathbb{R}^2)}+ 2 \lambda\|S\|_{H^1(\mathbb{R}^2)}\|h\|_{H^1(\mathbb{R}^2)}\notag\\
    &=L_6\|h\|_{H^1(\mathbb{R}^2)}\notag.
\end{align}
Here, $L_5, L_6 > 0$ are constants, and we apply H\"older’s inequality in the transformation. Therefore, $\Tilde{T}$ is the bounded linear operator and $F$ is Fr\'echet differentiable on $H^1(\mathbb{R}^2)$.
\end{proof}

\begin{thm}
    The gradient of $F$ at $ S $ is given as
\begin{align}
    \nabla F(S; \sigma_2) &= 2^8 \pi k_0^4 \sum_{m_1 = \pm 1} \sum_{m_2 = \pm 1} \mathbf{1}_{\{(p, q) \in \mathbb{R}^2 \, | \, f_{m_1, m_2}(p, q) \in K\}}(-x-k_0,-y) \notag \\
    &\quad \times\{A[S] \left( f_{m_1, m_2} \left( -x - k_0, - y \right)\right) - \sigma_2 \left( f_{m_1, m_2} \left( -x - k_0, - y \right) \right) \}\notag \\
    &\quad \times S\left( m_1  \left( x - 2 k_0, - y \right) \right) \notag \\
    &\quad \times | \Gamma \left( m_1, m_2, (x - 2k_0, - y), (x, y), f_{m_1, m_2} \left( x - k_0, - y\right) \right) |^2 +2\lambda S(x,y),
\end{align}
where $\mathbf{1}_{\{(p, q) \in \mathbb{R}^2 \, | \, f_{m_1, m_2}(p, q) \in K\}}$ is the indicator function.
\end{thm}

\begin{proof}
Let us consider the term $IV$ in \eqref{F_derivative}. By Proposition \ref{prop3},
\begin{align}
    IV = 2^8 \pi k_0^4\sum_{m_1 = \pm 1} \sum_{m_2 = \pm 1} \int_K \int_{f_{m_1, m_2}^{-1}(\omega)}& 
    \frac{\{A[S](\omega) - \sigma_2(\omega)\} |\Gamma(m_1, m_2, \omega, \bm{k_1}, \bm{k_2})|^2}{|\nabla f_{m_1, m_2}(p, q)|}\notag \\
    &\times S(m_1 \bm{k_1})h(m_2 \bm{k_2}) \, ds \, d\omega.\notag
\end{align}
The function $f_{m_1, m_2}$ is Lipschitz. We define $g(p, q)$ as the integrand of $IV$ with $\omega = f_{m_1, m_2}(p, q)$. Then,
\begin{align}
    g(p,q):=&\frac{\{A[S](f_{m_1, m_2}(p, q)) - \sigma_2(f_{m_1, m_2}(p, q))\} |\Gamma(m_1, m_2, f_{m_1, m_2}(p, q), \bm{k_1}, \bm{k_2})|^2}{|\nabla f_{m_1, m_2}(p, q)|}\notag \\
    &\times S(m_1 \bm{k_1})h(m_2 \bm{k_2}).\notag
\end{align}
Since $g(p, q)$ is Lebesgue integrable on $f_{m_1, m_2}^{-1}(K)$, using the coarea formula (see \citep[Theorem 3.11, p.139]{evans2015measure} and \citep[Theorem 3.2.12, pp.249--250]{federer2014geometric}), we obtain
\begin{align}
    IV &= 2^8 \pi k_0^4\sum_{m_1 = \pm 1} \sum_{m_2 = \pm 1}\int_K \int_{f_{m_1, m_2}^{-1}(\omega)}g(p, q)\, ds \, d\omega \notag\\
    &= 2^8 \pi k_0^4\sum_{m_1 = \pm 1} \sum_{m_2 = \pm 1} \int_{\{(p, q) \in \mathbb{R}^2 \, | \, f_{m_1, m_2}(p, q) \in K\}}g(p, q) |\nabla f_{m_1, m_2}(p, q)|\, dp \, dq\notag \\
    &=2^8 \pi k_0^4 \sum_{m_1 = \pm 1} \sum_{m_2 = \pm 1}\int_{(p, q) \in \mathbb{R}^2} \mathbf{1}_{\{(p, q) \in \mathbb{R}^2 \, | \, f_{m_1, m_2}(p, q) \in K\}}(p,q)\notag\\
    &\phantom{=}\times \big\{A[S] - \sigma_2\big\}|\Gamma|^2S(m_1 \bm{k_1}) h(m_2 \bm{k_2})\, dp \, dq.\notag
\end{align}
By the variable transformation $ m_2 \bm{k}_2 \mapsto (x, y) $ and the term $V$,
\begin{align}
    &F(S + h; \sigma_2) - F(S; \sigma_2) \notag\\
    &= 2 \int_K A'[S] h (A[S] - \sigma_2) \, d\omega + 2 \lambda \int_{\mathbb{R}^2} S h \, dp \, dq+o(\|h\|_{H^1(\mathbb{R}^2)})\notag\\
    &=\int_{(p, q) \in \mathbb{R}^2} \bigg[2^8 \pi k_0^4 \sum_{m_1 = \pm 1} \sum_{m_2 = \pm 1} \mathbf{1}_{\{(p, q) \in \mathbb{R}^2 \, | \, f_{m_1, m_2}(p, q) \in K\}}(-x-k_0,-y) \notag \\
    &\quad \times\{A[S] \left( f_{m_1, m_2} \left( -x - k_0, - y \right)\right) - \sigma_2 \left( f_{m_1, m_2} \left( -x - k_0, - y \right) \right) \}\notag \\
    &\quad \times S\left( m_1  \left( x - 2 k_0, - y \right) \right) \notag \\
    &\quad \times | \Gamma \left( m_1, m_2, (x - 2k_0, - y), (x, y), f_{m_1, m_2} \left( x - k_0, - y\right) \right) |^2 +2\lambda S(x,y)\bigg]h(x,y) \, dx \, dy\nonumber\\
    &\quad +o(\|h\|_{H^1(\mathbb{R}^2)})\notag.
\end{align}
From Proposition \ref{cor4}, we obtain
\begin{align}
    &F'(S; \sigma_2)h=
    \int_{(p, q) \in \mathbb{R}^2} \bigg[2^8 \pi k_0^4 \sum_{m_1 = \pm 1} \sum_{m_2 = \pm 1} \mathbf{1}_{\{(p, q) \in \mathbb{R}^2 \, | \, f_{m_1, m_2}(p, q) \in K\}}(-x-k_0,-y) \notag \\
    &\quad \times\{A[S] \left( f_{m_1, m_2} \left( -x - k_0, - y \right)\right) - \sigma_2 \left( f_{m_1, m_2} \left( -x - k_0, - y \right) \right) \}\notag \\
    &\quad \times S\left( m_1  \left( x - 2 k_0, - y \right) \right) \notag \\
    &\quad \times | \Gamma \left( m_1, m_2, (x - 2k_0, - y), (x, y), f_{m_1, m_2} \left( x - k_0, - y\right) \right) |^2 +2\lambda S(x,y)\bigg]h(x,y) \, dx \, dy\nonumber.
\end{align}
Consequently, using the Riesz representation theorem (as in \citep[Theorem 5.5, p.135]{brezis2010functional}), we can obtain
\begin{align}
    \nabla F(S; \sigma_2) &= 2^8 \pi k_0^4 \sum_{m_1 = \pm 1} \sum_{m_2 = \pm 1} \mathbf{1}_{\{(p, q) \in \mathbb{R}^2 \, | \, f_{m_1, m_2}(p, q) \in K\}}(-x-k_0,-y) \notag \\
    &\quad \times\{A[S] \left( f_{m_1, m_2} \left( -x - k_0, - y \right)\right) - \sigma_2 \left( f_{m_1, m_2} \left( -x - k_0, - y \right) \right) \}\notag \\
    &\quad \times S\left( m_1  \left( x - 2 k_0, - y \right) \right) \notag \\
    &\quad \times | \Gamma \left( m_1, m_2, (x - 2k_0, - y), (x, y), f_{m_1, m_2} \left( x - k_0, - y\right) \right) |^2 +2\lambda S(x,y).\notag
\end{align}
\end{proof}

Since the functional $ F(S; \sigma_2) $ satisfies Assumption \ref{assu1} by Corollary \ref{cor4}, Theorem \ref{the2} is applicable to the objective functional $\Theta(S) = \|A[S] - \sigma_2\|_{L^2(K)}^2 + \lambda \|S\|_{L^2(\mathbb{R}^2)}^2 + \alpha\Phi(S)$. 

Next, we investigate the stability of a functional $\Theta$, which combines Tikhonov regularization with nonnegative sparse regularization. Although the stability of Tikhonov regularization itself is well-established \cite{kirsch1996introduction, clason2021}, the stability analysis of functionals incorporating 
nonnegative sparse regularization has not yet been studied sufficiently. To address this gap, we therefore perform our analysis in a general setting.

Letting $X, Y$ be Hilbert spaces and letting $F$ be an operator from $X$ to $Y$, we discuss the stability of the functional
\[
\Theta(u) = \|F(u) - y^\delta\|_Y^2 + \lambda \|u\|_X^2 + \alpha\Phi(u),
\]
where $u \in X$. In the functional, $y^\delta \in Y$ denotes the observed data with noise of level $\delta$, i.e., it satisfies $\|y^\delta - y\|_Y \le \delta$, where $y$ represents the true data.
\begin{prop}
We assume that $F$ is Fr\'echet differentiable. Additionally, the following assumptions hold:
\begin{enumerate}
    \item\label{a1} $ F $ has a nonnegative minimum norm solution $u^\dagger$ such that there exists a constant $R > 0$ satisfying $\Phi(u^\dagger) < R$.
    \item\label{a2} $F'$ is Lipschitz continuous with a Lipschitz constant $L$.
    \item\label{a3} There exists $w \in Y$ such that $u^\dagger = F'(u^\dagger)^* w$ and $L \|w\|_Y < 1$.
\end{enumerate}
Let $c_1, c_2, c_3, c_4 > 0$ be constants satisfying $c_1 \sqrt{\delta} < \lambda < c_2 \sqrt{\delta}$ and $c_3 \delta < \alpha < c_4 \delta$. Then, the following stability estimate holds as 
\[
\|u^{\delta}_{\lambda,\alpha} - u^\dagger\|_X \leq c\delta^{\frac{1}{4},}
\]
for all $0 < \delta < 1$, where $u_{\lambda, \alpha}^\delta$ is the global minimizer of the functional $\Theta$, and $c > 0$ is a constant.
\end{prop}
\begin{proof}
    
Since the functional $\Theta(u)$ admits $u^{\delta}_{\lambda,\alpha}$ as its global minimizer (Theorem \ref{the2}), 
\begin{align*}
    \Theta(u^{\delta}_{\lambda,\alpha}) 
    &= \|F(u^{\delta}_{\lambda,\alpha}) - y^\delta\|_Y^2  + \lambda \|u^{\delta}_{\lambda,\alpha}\|_X^2  + \alpha \Phi(u^{\delta}_{\lambda,\alpha})\notag\\
    &\leq \|F(u^\dagger) - y^\delta\|_Y^2 + \lambda \|u^\dagger\|_X^2 + \alpha \Phi(u^\dagger) = \Theta(u^\dagger).\notag
\end{align*}
Using the inequality $\|u\|_X \leq \Phi(u)$ proven in \citep[Lemma 2.1, p.3]{Muoi2018}, we obtain
\begin{align*}
    &\|F(u^{\delta}_{\lambda,\alpha}) - y^\delta\|_Y^2  + \lambda \|u^{\delta}_{\lambda,\alpha}\|_X^2  + \alpha \|u^{\delta}_{\lambda,\alpha}\|_X \\
    \leq &\|F(u^{\delta}_{\lambda,\alpha}) - y^\delta\|_Y^2  + \lambda \|u^{\delta}_{\lambda,\alpha}\|_X^2  + \alpha \Phi(u^{\delta}_{\lambda,\alpha}) = \Theta(u^{\delta}_{\lambda,\alpha}).\notag
\end{align*}
Thus, 
\begin{align}
    &\|F(u^{\delta}_{\lambda,\alpha}) - y^\delta\|_Y^2  + \lambda \|u^{\delta}_{\lambda,\alpha}\|_X^2  + \alpha \|u^{\delta}_{\lambda,\alpha}\|_X \leq \Theta(u^{\delta}_{\lambda,\alpha})\leq \Theta(u^\dagger).\notag
\end{align}
Therefore, we have 
\begin{align}
    \label{pr61}
    \|F(u^{\delta}_{\lambda,\alpha}) - y^\delta\|_Y^2  + \lambda \|u^{\delta}_{\lambda,\alpha}\|_X^2  + \alpha \|u^{\delta}_{\lambda,\alpha}\|_X \leq \|F(u^\dagger) - y^\delta\|_Y^2 + \lambda \|u^\dagger\|_X^2 + \alpha \Phi(u^\dagger).
\end{align}
From
\begin{align}
    \|u^{\delta}_{\lambda,\alpha}\|_X^2 &= \|(u^{\delta}_{\lambda,\alpha}-u^\dagger) + u^\dagger\|^2_X\notag\\
    &=\|u^{\delta}_{\lambda,\alpha} -  u^\dagger\|_X^2 + 2\langle u^\dagger,u^{\delta}_{\lambda,\alpha} -  u^\dagger \rangle_X+\|u^\dagger\|_X^2\notag
\end{align}
and the triangle inequality
\begin{align}
    \|u^{\delta}_{\lambda,\alpha}\|_X \ge\|u^{\delta}_{\lambda,\alpha} -  u^\dagger\|_X-\|u^\dagger\|_X,\notag
\end{align}
we can transform the inequality \eqref{pr61} as
\begin{align}\label{pr62}
    &\|F(u^{\delta}_{\lambda,\alpha}) - y^\delta\|_Y^2  + \lambda\|u^{\delta}_{\lambda,\alpha}-u^\dagger\|_X^2+2\lambda\langle u^\dagger,u^{\delta}_{\lambda,\alpha} -  u^\dagger \rangle_X 
    + \lambda\|u^\dagger\|^2_X 
    + \alpha\|u^{\delta}_{\lambda,\alpha}-u^\dagger\|_X - \alpha\|u^\dagger\|_X\\
    &\leq \|F(u^\dagger) - y^\delta\|_Y^2 + \lambda \|u^\dagger\|_X^2 + \alpha \Phi(u^\dagger).\notag
\end{align}
Since the assumption \eqref{a1} implies $F(u^\dagger) = y$ and we have $\|y - y^\delta\|_Y \leq \delta$, it follows that $\|F(u^\dagger) - y^\delta\|_Y \leq \delta$. Thus, the right-hand side of \eqref{pr62} represents
\begin{align}\label{pr63}
    \|F(u^\dagger) - y^\delta\|_Y^2 + \lambda \|u^\dagger\|_X^2 + \alpha \Phi(u^\dagger)
    &\leq \delta^2+\lambda\|u^\dagger\|_X^2 + \alpha\Phi(u^\dagger).
\end{align}
In addition, from assumption \eqref{a1}, there exists a constant $R>0$ such that $\Phi(u^\dagger) < R$. Then, from \eqref{pr63}, 
\begin{align}\label{pr64}
    \|F(u^\dagger) - y^\delta\|_Y^2 + \lambda \|u^\dagger\|_X^2 + \alpha \Phi(u^\dagger)
    &\leq \delta^2+\lambda\|u^\dagger\|_X^2 + \alpha R. 
\end{align}
Combining \eqref{pr62} and \eqref{pr64},
\begin{align*}
    &\|F(u^{\delta}_{\lambda,\alpha}) - y^\delta\|_Y^2  + \lambda\|u^{\delta}_{\lambda,\alpha}-u^\dagger\|_X^2+2\lambda\langle u^\dagger,u^{\delta}_{\lambda,\alpha} -  u^\dagger \rangle_X 
    + \lambda\|u^\dagger\|^2_X 
    + \alpha\|u^{\delta}_{\lambda,\alpha}-u^\dagger\|_X - \alpha\|u^\dagger\|_X \\
    &\leq \delta^2+\lambda\|u^\dagger\|_X^2 + \alpha R.
\end{align*}
Therefore, we can obtain the following estimate of
\begin{align}
\label{pr65}
    &\|F(u^{\delta}_{\lambda,\alpha}) - y^\delta\|_Y^2  + \lambda\|u^{\delta}_{\lambda,\alpha}-u^\dagger\|_X^2 +\alpha\|u^{\delta}_{\lambda,\alpha}-u^\dagger\|_X\notag\\
    &\leq \delta^2 + \alpha R - 2\lambda\langle u^\dagger,u^{\delta}_{\lambda,\alpha} -  u^\dagger \rangle_X + \alpha \|u^\dagger\|_X \notag \\
    &< \delta^2 + 2\alpha R - 2\lambda\langle u^\dagger,u^{\delta}_{\lambda,\alpha} -  u^\dagger \rangle_X, 
\end{align}
for which we used the inequality $\|u^\dagger\|_X \leq \Phi(u^\dagger) < R$ and eliminated the term $\lambda\|u^\dagger\|^2_X$ from both sides of \eqref{pr65}. By the assumption \eqref{a3}, the right-hand side of \eqref{pr65} can be expressed as
\begin{align*}
    &\delta^2 + 2\alpha R -2\lambda\langle u^\dagger,u^{\delta}_{\lambda,\alpha} -  u^\dagger \rangle_X\notag\\
    &= \delta^2 + 2\alpha R -2\lambda\langle F'(u^\dagger)^*w,u^{\delta}_{\lambda,\alpha} -  u^\dagger \rangle_X\notag\\
    &= \delta^2 + 2\alpha R -2\lambda\langle w,F'(u^\dagger)(u^{\delta}_{\lambda,\alpha} -  u^\dagger) \rangle_Y\notag\\
    &\leq \delta^2 + 2\alpha R +2\lambda\|w\|_Y\|F'(u^\dagger)(u^{\delta}_{\lambda,\alpha} -  u^\dagger) \|_Y.
\end{align*}
Combining \eqref{pr65} and \eqref{pr66}, we have
\begin{align}\label{pr66}
    &\|F(u^{\delta}_{\lambda,\alpha}) - y^\delta\|_Y^2  + \lambda\|u^{\delta}_{\lambda,\alpha}-u^\dagger\|_X^2 +\alpha\|u^{\delta}_{\lambda,\alpha}-u^\dagger\|_X\notag\\
    &\leq \delta^2 + 2\alpha R +2\lambda\|w\|_Y\|F'(u^\dagger)(u^{\delta}_{\lambda,\alpha} -  u^\dagger) \|_Y. 
\end{align}
From assumption \eqref{a2} and the inequality (see \citep[Lemma 9.5, p.85]{clason2021}), 
\begin{align*}
    \|F(u^{\delta}_{\lambda,\alpha}) - F(u^\dagger) - F'(u^\dagger)(u^{\delta}_{\lambda,\alpha} - u^\dagger)\|_Y \leq \frac{1}{2}L\|u^{\delta}_{\lambda,\alpha} - u^\dagger\|_X^2. 
\end{align*}
Then using the triangle inequality, we have 
\begin{align*}
     \|F'(u^\dagger)(u^{\delta}_{\lambda,\alpha} - u^\dagger)\|_Y - \|F(u^{\delta}_{\lambda,\alpha}) - F(u^\dagger)\|_Y \leq \frac{1}{2}L\|u^{\delta}_{\lambda,\alpha} - u^\dagger\|_X^2.  
\end{align*}
Subsequently, we obtain 
\begin{align*}
     \|F'(u^\dagger)(u^{\delta}_{\lambda,\alpha} - u^\dagger)\|_Y  &\leq \frac{1}{2}L\|u^{\delta}_{\lambda,\alpha} - u^\dagger\|_X^2 + \|F(u^{\delta}_{\lambda,\alpha}) - F(u^\dagger)\|_Y \\
     &=\frac{1}{2}L\|u^{\delta}_{\lambda,\alpha} - u^\dagger\|_X^2 + \|F(u^{\delta}_{\lambda,\alpha}) - y^\delta + y^\delta -F(u^\dagger)\|_Y \\
     &\leq\frac{1}{2}L\|u^{\delta}_{\lambda,\alpha} - u^\dagger\|_X^2 + \|F(u^{\delta}_{\lambda,\alpha}) - y^\delta\|_Y + \|y^\delta -F(u^\dagger)\|_Y \\
     &\leq\frac{1}{2}L\|u^{\delta}_{\lambda,\alpha} - u^\dagger\|_X^2 + \|F(u^{\delta}_{\lambda,\alpha}) - y^\delta\|_Y + \delta. 
\end{align*}
Thus,
\begin{align}\label{pr67}
    \|F'(u^{\delta}_{\lambda,\alpha} - u^\dagger)\|_Y\leq\frac{1}{2}L\|u^{\delta}_{\lambda,\alpha} - u^\dagger\|_X^2 + \|F(u^{\delta}_{\lambda,\alpha})- y^\delta\|_Y + \delta.
\end{align}
Substituting \eqref{pr67} into \eqref{pr66} gives
\begin{align*}
    &\|F(u^{\delta}_{\lambda,\alpha}) - y^\delta\|_Y^2  + \lambda\|u^{\delta}_{\lambda,\alpha}-u^\dagger\|_X^2 +\alpha\|u^{\delta}_{\lambda,\alpha}-u^\dagger\|_X\notag\\
    &\leq \delta^2 + 2\alpha R +\lambda L \|w\|_Y\|u^{\delta}_{\lambda,\alpha}-u^\dagger\|_X^2 +2\lambda\|w\|_Y\|F(u^{\delta}_{\lambda,\alpha}) - y^\delta\|_Y +2\delta \lambda\|w\|_Y.
\end{align*}
By adding $ \lambda^2\|w\|_Y^2 $ to both sides and then transforming the equation, we obtain
\begin{align*}
    \big[&\|F(u^{\delta}_{\lambda, \alpha}) - y^\delta\|^2_Y - 2\lambda\|w\|_Y\|F(u^{\delta}_{\lambda,\alpha}) - y^\delta\|_Y + \lambda^2 \|w\|^2_Y \big]
    + \lambda \|u^{\delta}_{\lambda, \alpha} - u^\dagger\|_X^2 + \alpha \|u^{\delta}_{\lambda, \alpha} - u^\dagger\|_X \notag \\
    &\leq 2\alpha R + \lambda L\|u^{\delta}_{\lambda, \alpha} - u^\dagger\|_X^2 \|w\|_Y + (\lambda^2 \|w\|^2_Y + 2\delta \lambda\|w\|_Y + \delta^2).
\end{align*}
Thus,
\begin{align}
    \big[\|F(u^{\delta}_{\lambda, \alpha}) - y^\delta\|_Y - \lambda \|w\|_Y \big]^2 
    &+ \lambda \|u^{\delta}_{\lambda, \alpha} - u^\dagger\|_X^2 + \alpha \|u^{\delta}_{\lambda, \alpha} - u^\dagger\|_X \notag \\
    &\leq 2\alpha R + \lambda L\|u^{\delta}_{\lambda, \alpha} - u^\dagger\|_X^2 \|w\|_Y + (\lambda \|w\|_Y + \delta)^2.\notag
\end{align}
By omitting the terms $\big[\|F(u^{\delta}_{\lambda, \alpha}) - y^\delta\|_Y - \lambda \|w\|_Y \big]^2$ and $\alpha \|u^{\delta}_{\lambda, \alpha} - u^\dagger\|_X$, 
we can produce 
\begin{align*}
    \lambda (1 - L \|w\|_Y) \|u^{\delta}_{\lambda, \alpha} - u^\dagger\|_X^2 \leq 
    2\alpha R + (\lambda \|w\|_Y + \delta)^2.
\end{align*}
Thus, 
\begin{align*}
    \|u^{\delta}_{\lambda, \alpha} - u^\dagger\|_X \leq \sqrt{\frac{2\alpha R + (\lambda \|w\|_Y + \delta)^2}{\lambda (1 - L \|w\|_Y)}}.
\end{align*}
From the assumptions $ c_1 \sqrt{\delta} < \lambda < c_2\sqrt{\delta} $ and $ c_3 \delta < \alpha < c_4 \delta $, with $0 < \delta < 1$, the following result is obtained: 
\begin{align*}
    \|u^{\delta}_{\lambda, \alpha} - u^\dagger\|_X &\leq \sqrt{\frac{2c_4 \delta R + (c_2 \sqrt{\delta} \|w\|_Y + \delta)^2}{c_1 \sqrt{\delta} (1 - L \|w\|_Y)}}\\
    &\leq \sqrt{\frac{2c_4 \delta R + (c_2 \sqrt{\delta} \|w\|_Y + \sqrt{\delta})^2}{c_1 \sqrt{\delta} (1 - L \|w\|_Y)}}\\
    &=\sqrt{\frac{2c_4 R + (c_2 \|w\|_Y + 1)^2}{c_1  (1 - L \|w\|_Y)}}\delta^\frac{1}{4}.
\end{align*}
Therefore, there exists a constant $c > 0$ such that 
\begin{equation*}
    \|u^{\delta}_{\lambda, \alpha} - u^\dagger\|_X \leq c \delta^{\frac{1}{4}}, 
\end{equation*}
which provides the desired estimate. 
\end{proof}

\section{Numerical experiments}\label{Numerical experiment}
As described in this section, using results from Section \ref{Preliminaries} and Section \ref{Numerical experiment}, we develop an algorithm for ocean spectrum estimation and than validate it. Furthermore, we perform wave height estimation using real data.

\subsection{Model wave spectrum and algorithm}
For numerical experiments, we use the ocean wave spectrum model based on \cite{kataoka2016}. The ocean wave directional spectrum $\hat{S}(f,\theta)$ is represented as the product of the frequency spectrum $K(f)$ and the angular distribution function $G(f,\theta)$, i.e., $\hat{S}(f,\theta) = K(f) G(f,\theta)$. The frequency spectrum $K(f)$ is based on the observational model proposed by Mitsuyasu \cite{Mitsuyasu1970}, defined as 
\[
K(f) = 0.257 H_s^2 T_{1/3}(T_{1/3} f)^{-5} \exp\left(-1.03 (T_{1/3} f)^{-4}\right).
\]
The directional spreading function $G(f,\theta)$ follows \cite{Mitsuyasu1975}, given as
\[
G(f,\theta) = G_0 \cos^{2s(f)}\left(\phi \left(\frac{\theta - \theta_0}{2}\right)\right).
\]
In the equations presented above, $ H_s $ and $ T_s $ respectively represent the significant wave height and the significant wave period. In addition, the angle $ \theta_0 $ represents the dominant wave direction.
The function $ \phi $ is expressed as 
\[
\phi(\theta) =
\begin{cases}
\theta + \pi, & \text{if } \theta < -\pi / 2, \\
\theta - \pi, & \text{if } \theta > \pi / 2, \\
\theta, & \text{otherwise},
\end{cases}
\]
where $-\pi < \theta \le \pi$ denotes the direction of the component waves. The exponent $ s(f) $ in the cosine function is a frequency-dependent directional spreading parameter dependent on the frequency $f$. Using $s_{\text{max}}$ and the peak frequency $ f_p = (1.05 T_s)^{-1} $, the exponent $s(f)$ is expressed as
\[
s(f) =
\begin{cases}
s_{\text{max}} \left( \frac{f}{f_p} \right)^5, & \text{if } f \leq f_p, \\
s_{\text{max}} \left( \frac{f}{f_p} \right)^{-2.5}, & \text{if } f > f_p.
\end{cases}
\]
In addition, the term $G_0$ is the normalization parameter expressed as
\[
G_0 = \left[ \int_{-\pi}^\pi \cos^{2s(f)} \left( \phi \left( \frac{\theta - \theta_0}{2} \right) \right) d\theta \right]^{-1}.
\]

We convert the ocean wave directional spectrum $ \hat{S}(f, \theta) = K(f) G(f, \theta) $ to the wave spectrum using the dispersion relation. The transformation is given as
\[
S(p, q) = \frac{1}{k} \frac{\partial \omega}{\partial k} \hat{S}(\omega, \theta),
\]
where $ \omega $ represents the angular frequency of the component waves, and $ k $ denotes the wavenumber \cite{kataoka2016}. We use this transformed spectrum $S(p,q)$ for numerical experimentation.
\begin{algorithm}
    \caption{Wave spectrum reconstruction algorithm}
    \label{algorithm1}
    \begin{algorithmic}[1]
        \State Set the initial wave spectrum $ S_0 $ such that $ \Phi(S_0) < +\infty $.
        \For{$n = 0,1,2,\dots$}
            \State Compute $S_{n+1} = \mathbb{P}_{t\alpha}(S_n - t\nabla F(S_n; \sigma_2))$, where $t$ is the step size.
            \While{$\Theta(S_{n+1}) > \Theta(S_n)$}
                \State Update $t \gets t \times \mu$, where $\mu \in (0, 1)$ is a constant to adjust the step size.
                \State Recompute $S_{n+1} = \mathbb{P}_{t\alpha}(S_n - t\nabla F(S_n; \sigma_2))$.
            \EndWhile
            \State Update $S_n \gets S_{n+1}$.
            \If{$\Theta(S_n)$ converges to a specific value or  $|\Theta(S_{n+1}) -\Theta(S_n)|$ becomes smaller than a certain threshold,}
                \State Terminate the algorithm.
            \EndIf
        \EndFor
    \end{algorithmic}
\end{algorithm}

We construct Algorithm \ref{algorithm1} based on the gradient method, referring to the algorithm proposed by \cite{Muoi2018}. We set the initial value $S_0$ as the wave spectrum generated by the model. Although we employ a discretized wave spectrum for the related computations, the algorithm requires spectral values at points that do not coincide with the computational grid. Therefore, in addition to the discretized spectrum, we use its interpolated version, obtained via spline interpolation, to provide these intermediate values. Since the wave spectrum asymptotically approaches zero at sufficiently distant regions, we assume that its values vanish outside the computational domain for numerical purposes.

\subsection{Validation of the algorithm}
For numerical experiments, we use the model wave spectrum generated using the method described above. This model spectrum is regarded as the true data. The second-order Doppler spectrum computed from \eqref{main} serves as the observed data. Additionally, perturbations at various levels, generated using uniform random numbers, are added to the true wave spectrum. The resulting perturbed spectra are then used as initial values for the algorithm. Through these numerical experiments, we evaluate how accurately the true wave spectrum can be reconstructed. To assess the accuracy, we consider the relative errors (error rates) of the wave spectrum and the significant wave height. The error rate of the wave spectrum is defined as
\begin{equation*}
\frac{\| S_{\text{true}} - S \|_{L^2(\mathbb{R}^2)}}{\| S_{\text{true}} \|_{L^2(\mathbb{R}^2)}} \times 100,
\end{equation*}
where $S_{\text{true}}$ represents the true wave spectrum, and where $S$ represents the wave spectrum under consideration. Similarly, the error rate of the significant wave height is expressed as
\begin{equation*}
\frac{| h_{\text{true}} - h |}{| h_{\text{true}} |} \times 100,
\end{equation*}
where $h_{\text{true}}$ represents the significant wave height computed from the true wave spectrum, and where $h$ denotes that computed from the wave spectrum $S$. For this study, we present the results obtained with a 40\% perturbation added to the true wave spectrum. The parameters are set as $\lambda = 10^{-3}$ and $\alpha = 10^{-6}$.

\begin{figure}[h]
    \centering

    \begin{minipage}{0.45\textwidth} 
        \centering
        \includegraphics[width=0.85\linewidth]{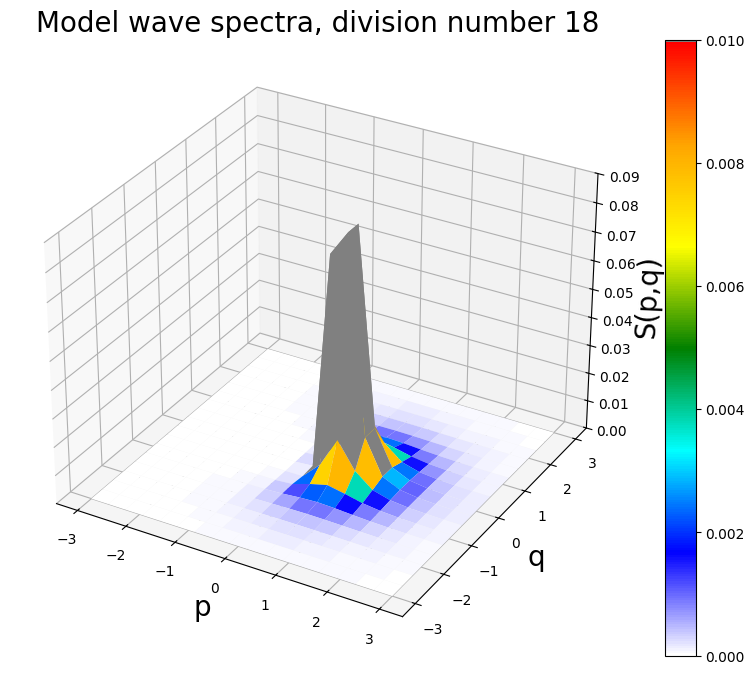}
        \subcaption*{($A$) True wave spectrum}
        \label{true}
    \end{minipage}
    \hfill
    \begin{minipage}{0.45\textwidth}  
        \centering
        \includegraphics[width=0.9\linewidth]{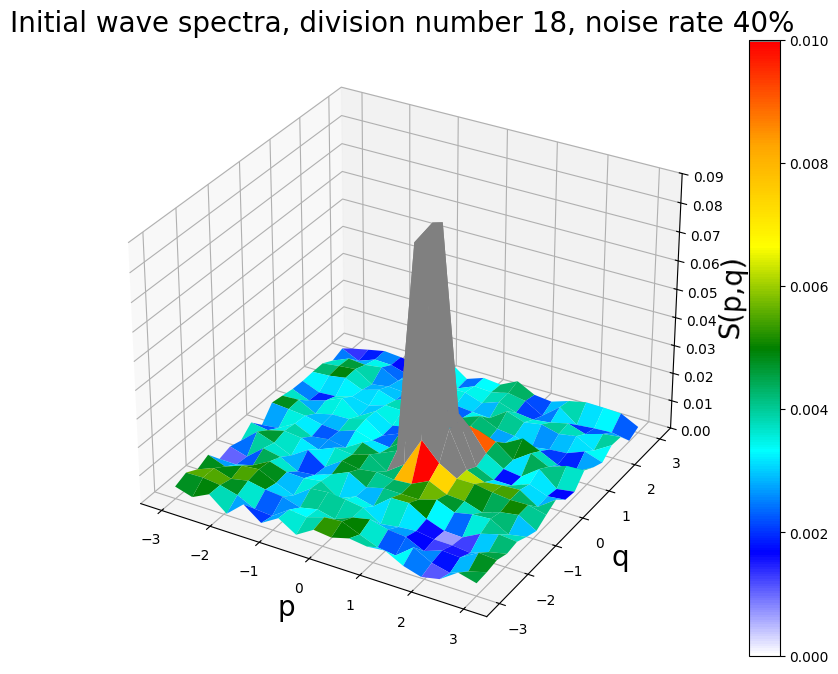}
        \subcaption*{($B$) Wave spectrum with perturbation}
        \label{noise}
    \end{minipage}

    \vspace{0.01\textwidth}

    \begin{minipage}{0.45\textwidth}
        \centering
        \includegraphics[width=0.9\linewidth]{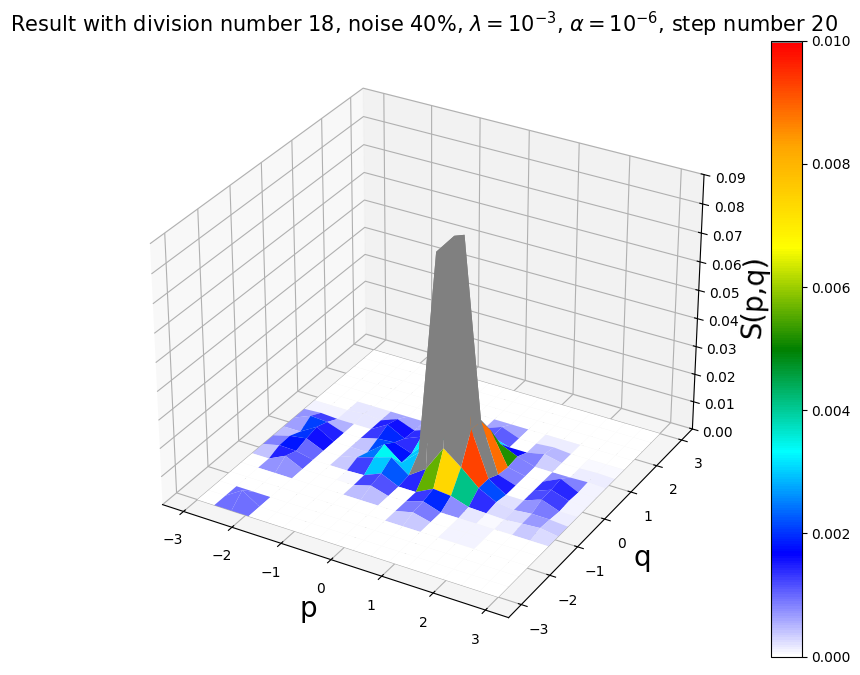}
        \subcaption*{($C$) Reconstructed wave spectrum}
        \label{result}
    \end{minipage}
    \hspace{0.04\textwidth}
    \begin{minipage}{0.49\textwidth}
        \centering
        \includegraphics[width=0.92\linewidth]{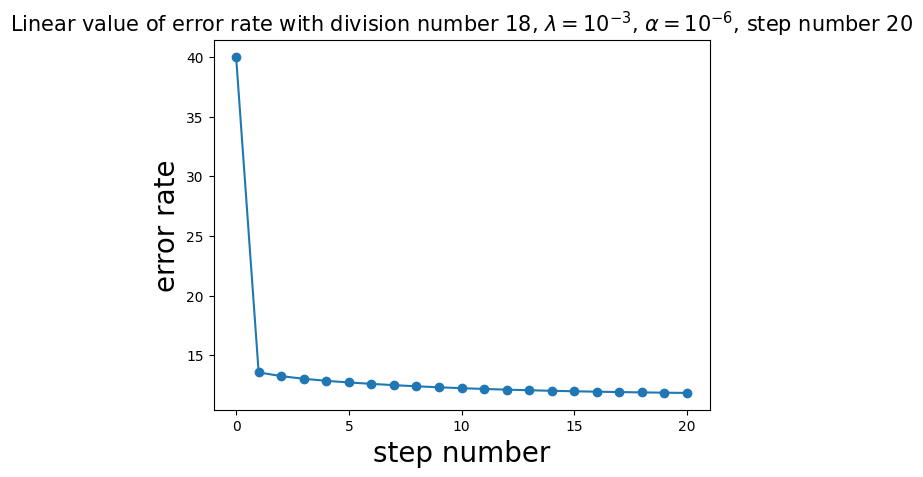}
        \subcaption*{($D$) Transition of wave spectrum error rate}
        \label{error_rate}
    \end{minipage}

    \vspace{0.02\textwidth}

    \begin{minipage}{0.48\textwidth}
        \centering
        \vspace{0.05\textwidth}
        \includegraphics[width=0.96\linewidth]{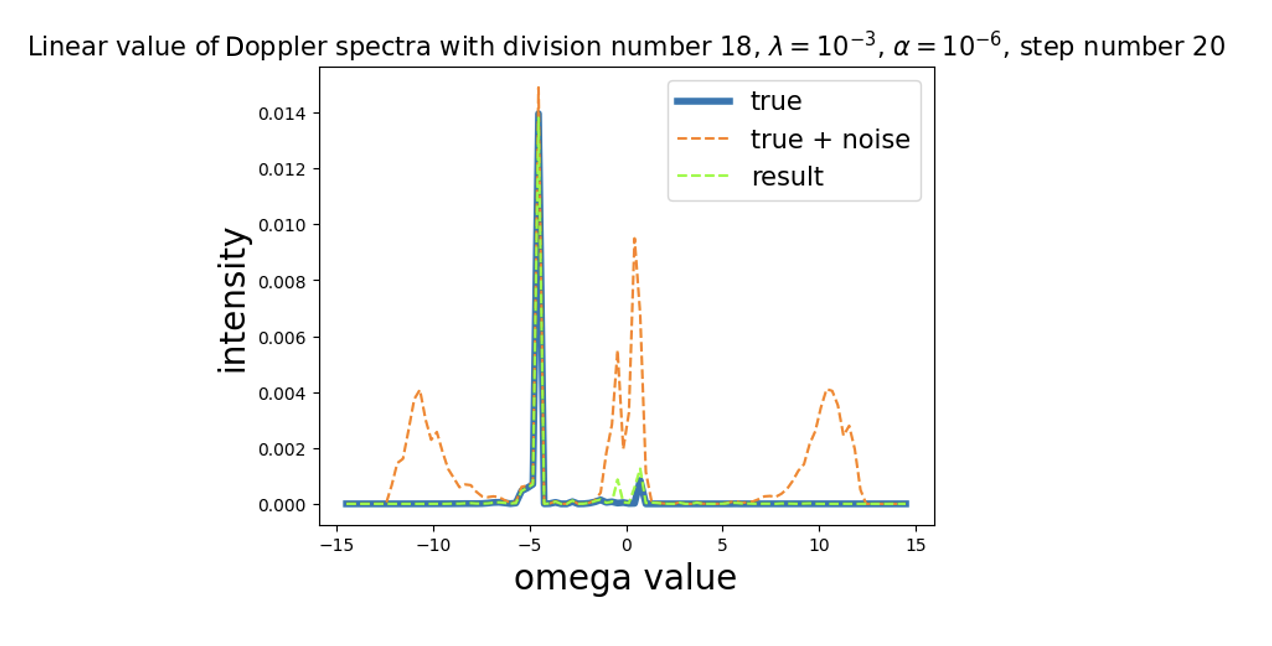}
        \subcaption*{($E$) Comparison of second-order Doppler spectra}
        \label{2dop}
    \end{minipage}
    \begin{minipage}{0.45\textwidth}
        \centering
        \includegraphics[width=1.05\linewidth]{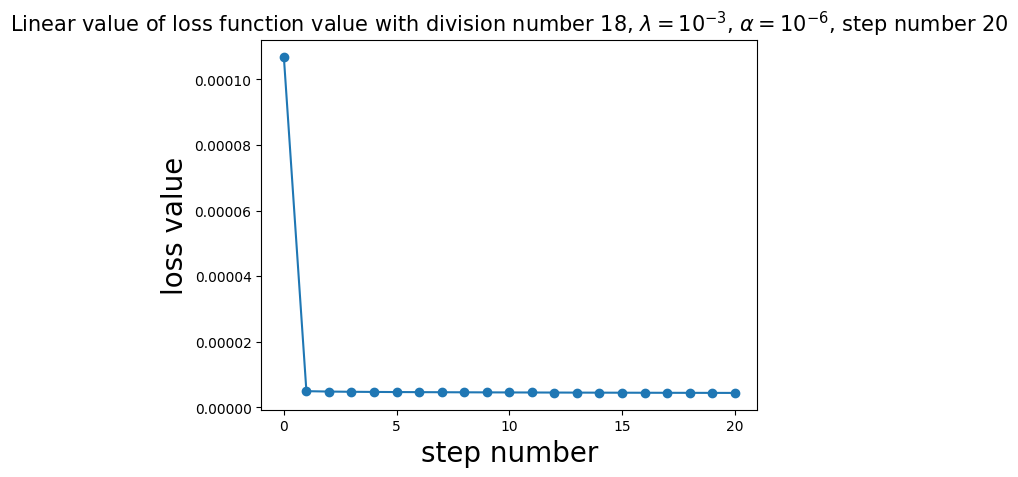}
        \subcaption*{($F$) Transition of the objective function}
        \label{loss}
    \end{minipage}
    
    \caption{Experimentally obtained results with 40\% relative error perturbation added.}
    \label{fig:four_images}
\end{figure}
\noindent

Figure \ref{fig:four_images}($A$) shows the true wave spectrum. Figure \ref{fig:four_images}($B$) displays the initial wave spectrum with perturbation. The true wave spectrum is generated with a significant wave height of 1 m, a significant wave period of 3 s, and a dominant wave direction of 0 rad, meaning that the waves travel along the ocean radar's line of sight. The wave spectrum coloration uses a value range of 0--0.01, where value 0 is shown in white. Areas with values exceeding 0.01 are shown as gray because no significant changes occur before and after the experiment.

Results obtained after 20 iterations of the algorithm are also presented in Figure \ref{fig:four_images}. Figure \ref{fig:four_images}($C$) portrays the reconstructed wave spectrum. It is readily apparent that the perturbation around the significant peaks has been reduced considerably, resulting in a shape closer to the true data. Figure \ref{fig:four_images}($D$) depicts the transition of the relative error rate, showing that although 40\% perturbation was initially added, the final error rate decreases to approximately 12\%. Figure \ref{fig:four_images}($E$) presents comparison of the second-order Doppler spectra derived from the true data (blue solid line), the perturbed spectrum (orange dashed line), and the reconstructed wave spectrum (green dashed line). It is evident that the Doppler spectra derived from the original and reconstructed spectra are almost identical. Additionally, Figure \ref{fig:four_images}($F$) shows that the value of the objective function converges to a certain constant.

\begin{figure}[t]
    \centering

    \begin{minipage}{0.45\textwidth} 
        \centering
        \includegraphics[width=\linewidth]{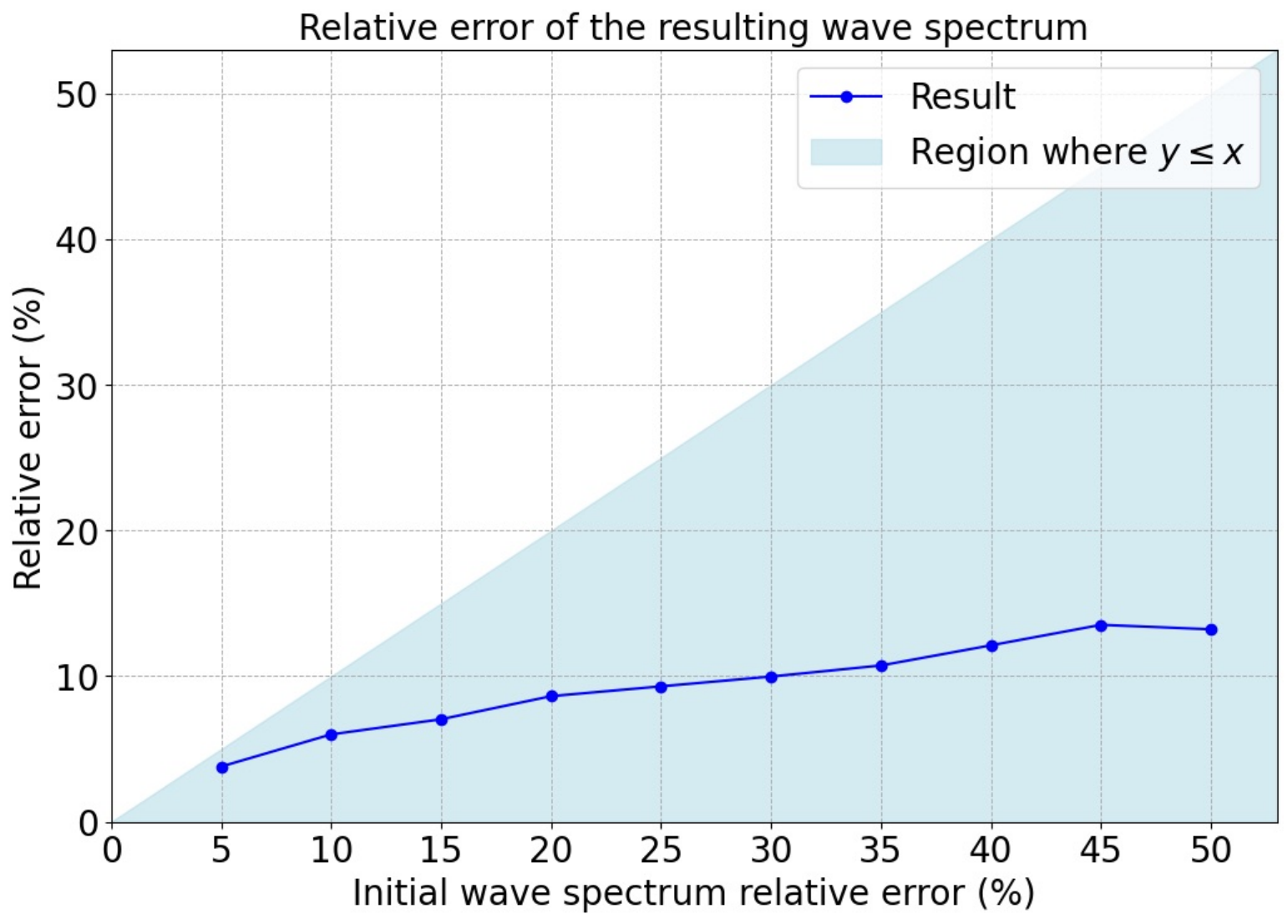} 
        \subcaption{Relative error of reconstructed wave spectrum}
        \label{fig:first_image}
    \end{minipage}
    \hspace{0.05\textwidth}
    \begin{minipage}{0.45\textwidth} 
        \centering
        \includegraphics[width=\linewidth]{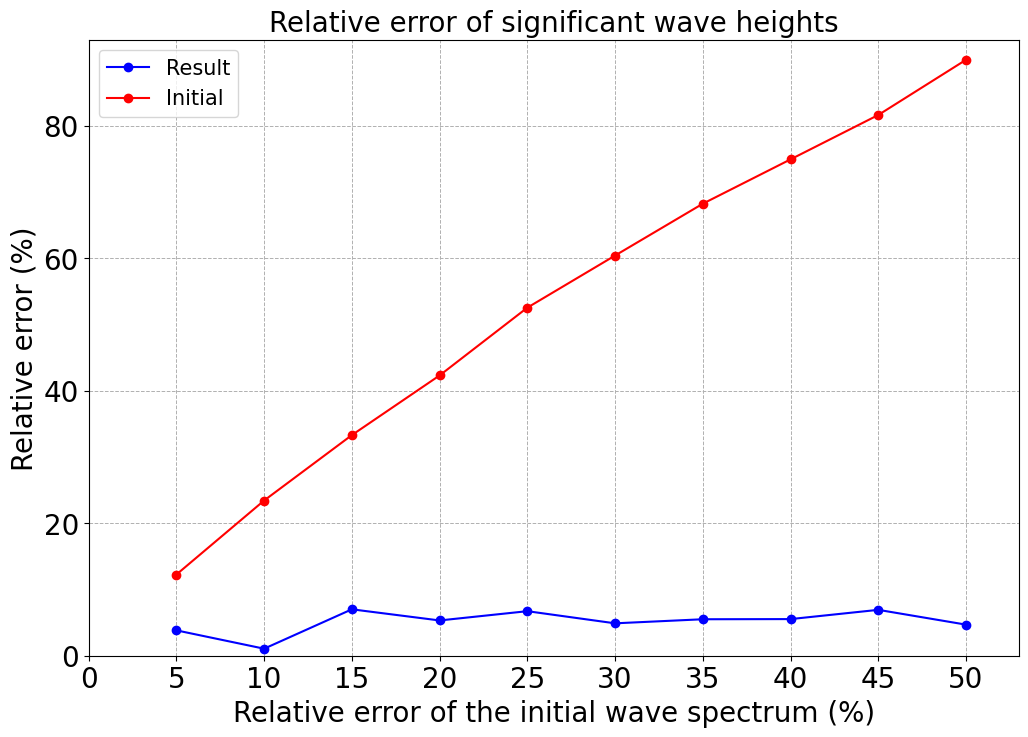} 
        
        \subcaption{Relative error of significant wave heights}
        \label{fig:second_image}
    \end{minipage}
    \caption{Results with perturbation intensity of 5\%--50\%.}
    \label{fig:two_images}
\end{figure}

Figure \ref{fig:two_images} presents the experimentally obtained results when the perturbation magnitude is changed from 5\% to 50\%. The left panel of Figure \ref{fig:two_images}($A$) shows the relative error of the reconstructed wave spectrum, whereas the right panel of Figure \ref{fig:two_images}($B$) depicts the transition of the relative error of the significant wave heights. In the graph showing the relative error of the wave spectrum, the horizontal axis represents the perturbation magnitude, and the vertical axis shows the resulting error rate. The light blue shaded area corresponds to the region where $y \leq x$, indicating that the error rate decreases. The blue graph lies within this region, showing that the maximum error rate is suppressed to approximately 13\%.

For the relative error of the significant wave heights (Figure \ref{fig:two_images}($B$)), the horizontal axis represents the perturbation magnitude, whereas the vertical axis indicates the error rate. The red graph shows the error rate estimated from the initial wave spectrum. The blue graph shows the error rate derived from the reconstructed wave spectrum. The significant wave heights estimated from the reconstructed wave spectrum are closer to the true values than those derived from the initial spectrum, with the error rate reduced to approximately 7\%. These results demonstrate that the wave spectrum and significant wave heights can be estimated with high accuracy.

For the next experiment, a model wave spectrum is generated. Noise levels of 0.5\%, 1\%, and 5\% are added to the second-order Doppler spectrum calculated from it. These noisy Doppler spectra are treated as the observed data. Deviation from the true wave spectrum and the error of the estimated significant wave heights are analyzed. We set the initial value of the algorithm to the true wave spectrum for these experiments.

We set the parameters as $ \lambda = 10^{-3} $ and $ \alpha = 10^{-6} $. The relative errors obtained from the results after 10 iterations are presented in Table \ref{tab:errors}. From these results, it can be inferred that the algorithm is stable. Furthermore, in terms of significant wave height estimation, the deviation remains minimal, thereby indicating the robustness of the algorithm.

\begin{table}[h]
\centering

\begin{tabular}{ccc}
\toprule
Noise (\%) & Wave spectrum error (\%) & Significant wave height error (\%) \\
\midrule
0.5 & $6.4 \times 10^{-5}$ & $1.9 \times 10^{-3}$ \\
1   & $3.4 \times 10^{-5}$ & $2.9 \times 10^{-4}$ \\
5   & $7.2 \times 10^{-6}$ & $4.3 \times 10^{-4}$ \\
\bottomrule
\end{tabular}
\caption{Relative errors for the wave spectrum and the significant wave height}
\label{tab:errors}
\end{table}

\subsection{Wave height estimation using real data}
Finally, we estimate the significant wave heights using the observed second-order Doppler spectra. Data were obtained for Murotsu Bay, located in Muroto city, Kochi prefecture, Japan. The observation period was 9:56 AM to 11:56 AM on May 11, 2023, with measurements taken at 20 min intervals.

The algorithm requires an initial guess, which is generated using significant wave heights and significant wave periods observed in Murotsu Bay on May 11, 2022, from 9:40 AM at 20 min intervals. These values were provided by the Nationwide Ocean Wave information network for Ports and HArbourS website (NOWPHAS) \cite{nowphas}. We assume the wave propagation direction to be north (180 degrees) for the model. We set the regularization parameter $\alpha$ to $1.0 \times 10^{-8}$. Additionally, we adjust $\lambda$ to match the true wave height value closely at 9:56. From numerical experiments, $\{\|\sigma_2 - A[S_\text{initial}]\|_2 / \|\sigma_2\|\} \times 10^{4}$ was found to be a candidate for $\lambda$. Therefore, we adopt it.
\begin{figure}[tb]
    \centering
    \includegraphics[width=0.8\linewidth]{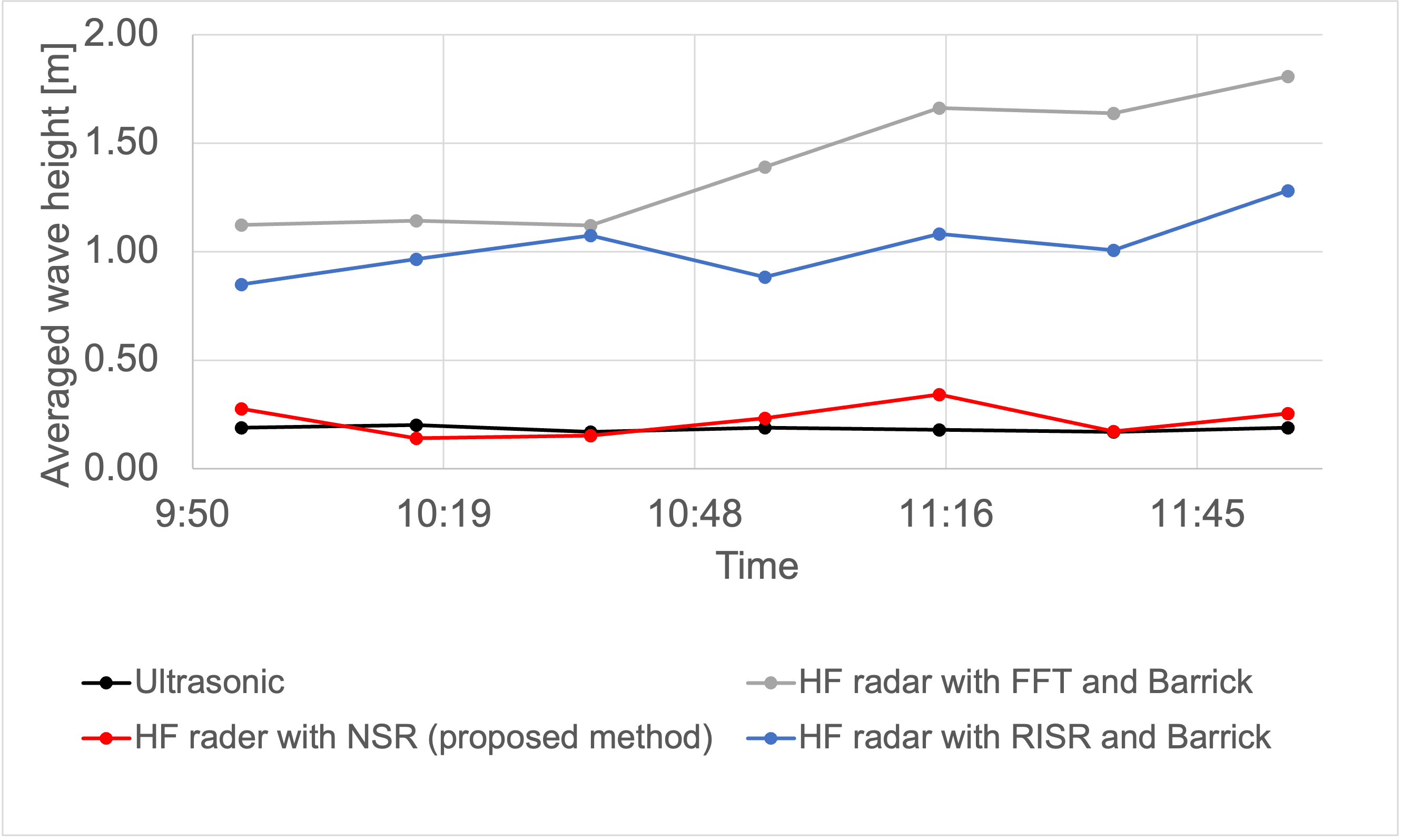} 
    \caption{Summary of estimated wave heights. The black line represents the true values, whereas the red line represents the wave heights estimated using the proposed method. The other lines correspond to estimations obtained using methods based on the Barrick method.}
    \label{result2}
\end{figure}

We compare the results obtained using the method proposed herein with the ultrasonic measurement \cite{Sener2023}, which we regard as true significant wave heights, and with two earlier estimation results obtained using the Barrick method with HF radar data combined with Doppler spectrum refinement techniques, such as fast Fourier transform (FFT) and re-iterative super-resolution (RISR) presented in \cite{kameda2023hf}. Figure \ref{result2} shows the ultrasonic measurement data as the black line, whereas the red line represents results obtained using the proposed method. The other two lines correspond to earlier estimations derived from the combined methods described above. The significant wave heights obtained using the ultrasonic measurement and using the methods combining the Barrick approach with FFT or RISR are based on estimated values reported by \cite{kameda2023hf}. From the comparison presented in Figure \ref{result2}, it can be inferred that the proposed method has potential to achieve wave height estimation with higher accuracy. However, this result should be interpreted with caution because the regularization parameter $\lambda$ was adjusted using the true wave height at 9:56. Additionally, the high accuracy of estimation might be attributed to the uniform wave conditions. Therefore, applying the method to Doppler data with various wave conditions is necessary to evaluate the effectiveness of this method.

\section{Conclusions}\label{Conclusions}
We showed the line integral expression of the relation between the second-order Doppler spectrum and the wave spectrum. Based on this form, we proved that the Doppler spectrum belongs to $L^2$ when the wave spectrum is a member of $H^1$. Using these facts, we developed a wave spectrum estimation algorithm based on the nonnegative sparse regularization technique, derived the explicit gradient form of the associated Tikhonov functional, and then discussed its stability. Numerical experiments demonstrated that this algorithm has potential to achieve high estimation accuracy of significant wave heights. Our study is the first to present the mathematically rigorous treatment of the inverse problem for estimating the wave spectrum.

In future work, after evaluating the effectiveness of the wave height estimation algorithm by applying it to real data under various wave conditions, we shall seek improved methods to determine the parameters $\lambda$ and $\alpha$.


\section*{Acknowledgements}

This study was supported by a collaborative research agreement between Tohoku University and the Information Technology R\&D Center of Mitsubishi Electric Corporation. This work was also supported by JST SPRING (JPMJSP2114, K.W.).

\bibliographystyle{plain}

\end{document}